\documentclass[11pt]{article}

 \usepackage{amssymb, amsmath}
 \usepackage{epsfig}
 \usepackage{graphicx}
 \usepackage{pict2e}
 \usepackage[breaklinks=true]{hyperref}
\hypersetup{
colorlinks=true,
allcolors=black
}
 
 \newtheorem{Lemma}{Lemma}
 \newtheorem{Proposition}[Lemma]{Proposition}

\newcommand{\len}{\mathrm{len}}
\newcommand{\ent}{\mathrm{ent}}
\newcommand{\ave}{\mathrm{ave}}
 \newcommand{\GG}{\mbox{${\mathcal G}$}} 
 \newcommand{\HH}{\mbox{${\mathcal H}$}} 
  
 \renewcommand{\SS}{\mbox{${\mathcal S}$}}
  \newcommand{\EE}{\mbox{${\mathcal E}$}}
  \newcommand{\TT}{\mbox{${\mathcal T}$}}
  \newcommand{\MM}{\mbox{${\mathcal M}$}}
 
 \newcommand{\sfrac}[2]{{\textstyle\frac{#1}{#2}}}
 \newcommand{\proof}{{\bf Proof.\ }}

  \newcommand{\Ind}{{\rm 1\hspace{-0.90ex}1}}
  \newcommand{\bX}{{\mathbf X}}
    \newcommand{\bA}{{\mathbf A}}
    \newcommand{\bB}{{\mathbf B}}

 \newcommand{\bb}{\mathbf{b}}
 \newcommand{\ba}{\mathbf{a}}
 
 \newcommand{\bp}{\mathbf{p}}
 \newcommand{\bw}{\mathbf{w}}
 \newcommand{\Names}{\mathbf{Names}}
 \newcommand{\bL}{\Names}
   \newcommand{\qed}{\ \ \rule{1ex}{1ex}}
    \newcommand{\eps}{\varepsilon}
  \newcommand{\Ex}{{\mathbb E}}
\renewcommand{\Pr}{\mathbb{P}}

\newcommand{\GGunl}{\GG^{\mbox{{\footnotesize unl}}}}
\newcommand{\comment}[1]{}
\def\ER{Erd\H{o}s-R\'enyi}
\newcommand{\bal}[1]{\begin{align*}#1\end{align*}}
\newcommand{\baln}[1]{\begin{align}#1\end{align}}
\newcommand{\ta}[1]{^{(#1)}}
 \begin{document}

\title{Entropy of Some Models of Sparse Random Graphs With Vertex-Names}
 \author{David J. Aldous\thanks{Department of Statistics,
 367 Evans Hall \#\  3860,
 U.C. Berkeley CA 94720;  aldous@stat.berkeley.edu;
  www.stat.berkeley.edu/users/aldous.  Aldous's research supported by
 N.S.F Grant DMS-1106998. }
\and Nathan Ross\thanks{Department of Statistics,
 367 Evans Hall \#\  3860,
 U.C. Berkeley CA 94720;  ross@stat.berkeley.edu.}
}

 \maketitle

 \begin{abstract}
Consider the setting of sparse graphs 
on $N$  vertices, where the vertices have distinct 
``names", which are strings of length $O(\log N)$  from a fixed finite alphabet.  For many natural probability models, the entropy grows 
as $c N \log N$ for some model-dependent rate constant $c$. 
The mathematical content of this paper is the (often easy) calculation of $c$ for a variety of models, in particular for various standard random graph models adapted to this setting.
Our broader purpose is to publicize this particular setting as a natural 
setting for future theoretical study of data compression for graphs, and (more speculatively) for discussion of 
 unorganized versus organized complexity.
 \end{abstract}
 \vspace{0.1in}



 \vspace{0.4in}

 {\em Short title:}
Entropy of Sparse Random Graphs

 \newpage
 \section{Introduction}
 \label{sec-INT}

The concept {\em entropy} arises across a broad range of topics
within the mathematical sciences, with different nuances and applications. 
There is a substantial literature (see Section~\ref{sec:concept}) on topics linking entropy and graphs, 
 but our focus seems different from these.
In this paper we use the word only with its most elementary meaning: 
for any probability distribution $\bp = (p_s)$ on any finite set $S$, 
its {\em entropy} is the number 
\begin{equation}
 \ent(\bp) = - \sum_s p_s \log p_s .
\label{ent-def}
\end{equation}
For an $S$-valued random variable $X$ we abuse notation by writing 
$\ent(X)$ for the entropy of the distribution of $X$.

Consider an $N$-vertex undirected graph. 
Instead of the usual conventions about vertex-labels 
(unlabelled; labeled by a finite set independent of $N$; labeled by integers 
$1,\ldots,N$) our convention is that there is a fixed 
(i.e. independent of $N$) alphabet $\bA$ of size $2 \le A < \infty$ and that each vertex
has a different ``name", which is a length-$O(\log N)$ string 
$\ba = (a_1,\ldots,a_m)$ of letters from $\bA$.

We will consider  probability distributions over such graphs-with-vertex-names, in the $N \to \infty$ ``sparse graph limit"  where the number of edges is $O(N)$.
In other words we study random graphs-with-vertex-names $\GG_N$ whose average degree is $O(1)$.
In this particular context (see Section~\ref{sec:TSU} for discussion) one expects  
that the entropy should grow as 
\begin{equation}
 \ent(\GG_N) \sim c N \log N,
\label{c-def}
\end{equation}
 where $c$ is thereby interpretable as 
an ``entropy rate".  
 Note the intriguing curiosity that 
the numerical value of the entropy rate $c$ does not depend on the base 
of the logarithms, because there is a ``log" on both sides of the 
definition (\ref{c-def}), and indeed we will mostly avoid specifying the base.

In Section~\ref{sec:easy} we define and analyze a variety of models for which calculation of entropy rates is straightforward. 
In Section~\ref{sec:hard} we study one more complicated model.  This is the {\em mathematical} content of the paper.
Our motivation for studying entropy in this specific setting is discussed verbally in Section~\ref{sec:BP}, and this discussion 
is the main {\em conceptual} contribution of the paper.
The discussion is independent of the subsequent mathematics but may be helpful in formulating interesting probability models for future study.   
Section~\ref{sec:backg} gives some (elementary) technical background. 
Section \ref{sec:final} contains final remarks and open problems.

\section{Remarks on data compression for graphical\\ structures}
\label{sec:BP}
The well-known textbook~\cite{MR2239987} provides an account of the classical Shannon setting of data compression for sequential data, 
motivated by English language text modeled as a stationary random sequence.  
What is the analog for graph-structured data?

This is plainly a vague question.  
Real-world data rarely consists only of the abstract mathematical structure -- unlabelled vertices and edges -- of a graph; typically  a considerable  amount of context-dependent extra information is also present. 
Two illustrative examples:\\
(i) Phylogenetic trees on species; here part of the data is the names of the species 
 and the names of clades; \\
(ii) Road networks; here part of the data is the names or numbers of the roads 
and some indication of the locations  where roads meet.

\smallskip
\noindent 
Our setting is designed as one simple abstraction of ``extra information",
in which the (only) extra information is 
the ``names" attached to vertices. 
Note 
 that in many examples one expects  some association between the names and the graph structure, in that the names of two vertices which are adjacent 
will on average be ``more similar" in some sense than 
the names of two non-adjacent vertices. 
This is very clear in the phylogenetic tree example, because of the 
genus-species naming convention. 
So when we study toy probability models later, we want models featuring 
such association. 

Let us remind the reader of two fundamental facts from 
information theory \cite{MR2239987}.

\noindent
(a) In the general setting (\ref{ent-def}), there 
there exists a coding (e.g.\ Huffman code) $f_\bp: S \to \bB$ 
such that, for $X$ with distribution $\bp$, 
\[ \ent(\bp) \le \Ex  \ \len (f_\bp(X)) \le \ent(\bp) + 1  \]
and no coding can improve on the lower bound. 
Here $\bB$ denotes the set of finite binary strings 
$\bb = b_1b_2 \ldots b_m$ and $\len (\bb) = m$ denotes the length of a string and entropy is computed to base $2$.  
Recall that a coding is just a $1 - 1$ function.

\noindent
(b) In the classical Shannon setting, one considers   
a stationary ergodic sequence $\bX = (X_i)$ with values in a finite alphabet. 
Such a sequence has an {\em entropy rate} 
\[
H := \lim_{k \to \infty} k^{-1} \ent(X_1,\ldots,X_k) .
\]
Moreover there exist coding functions $f$ (e.g.\ Lempel-Ziv) which are {\em universal} in the sense that for every such stationary ergodic sequence, 
\[ \lim_{m \to \infty} 
m^{-1}  \Ex  \ \len (f(X_1, \ldots,X_m)) = H . \]

\noindent
The important distinction is that in (a) the coding function $f_\bp$ depends on the distribution of $\bX$ but in (b) the coding function $f$ is a function on finite sequences which does not depend on the distribution of $\bX$.

In our setting of graphs with vertex-names we can in principle apply (a), but it will typically be very unrealistic to imagine that observed real-world data is a realization from  some {\em known} probability distribution on such graphs.
At the other extreme, for many reasons one cannot expect there to exist, in our setting, 
``universal" algorithms analogous to (b).  
For instance, the vertex-names $(\ba, \ba^*)$  across some edges might be related by a deterministic cryptographic function.
Also note it is difficult to imagine a definition analogous to ``stationary" in our setting.
So it seems necessary to rely on heuristic algorithms for compression, 
where {\em heuristic} means only that there is no good theoretical guarantee on compressed length.
One could of course compare different heuristic algorithms at an empirical level by testing them on real-world data. 
As a theoretical complement, one could test an algorithm's efficiency  by trying to prove  that, for some wide range of  qualitatively different probability models for $\GG_N$, the 
algorithm behaves optimally in the sense of compressing to
mean length $(c + o(1)) N \log N$ where $c$ is the entropy rate 
(\ref{c-def}).
And the contribution of this paper is to provide a collection of
probability models for which we know the numerical value of $c$.

\subsection{Remarks on the technical setup}
\label{sec:TSU} 
The discussion above did not involve
two extra assumptions made in Section \ref{sec-INT}, 
that the graphs are sparse and that 
the length of names is $O(\log N)$ 
(note the length must be at least order $\log N$ to allow the names
to be  distinct).
These extra assumptions create a more focussed setting for data compression that is
mathematically interesting for two reasons.  
If the entropies of the two structural components -- the unlabelled graph, and  the set of names --  were of different orders,  then only the larger one would be important; 
but these extra assumptions make both entropies be of the same order,  $N \log N$.  
So both of these two structural components and their association become relevant for compression.
A second, more technical, reason is that natural models 
of sparse random graphs $\GG_n$ invariably have a well-defined limit 
$\GG_\infty$  in the sense of 
{\em local weak convergence} \cite{MR2023650,MR2354165}
of unlabelled graphs, and the limit $\GG_\infty$  automatically has a property 
{\em unimodularity} directly analogous to stationarity for 
random sequences.  
This addresses part of the  ``difficult to imagine a definition analogous to stationary in our setting" issue raised above, but it remains difficult  
to extend this notion to encompass the vertex-names.

\subsection{Related work}
\label{sec:concept}
We have given a verbal argument that the Section \ref{sec-INT} setting of sparse graphs with vertex-names 
is a worthwhile setting for future theoretical study of data compression in graphical structures.
It is perhaps surprising that this precise setting has apparently not been 
considered previously. 
The large literature on what is called ``graph entropy", recently surveyed 
in \cite{MR2737460}, deals with statistics of a single unlabelled graph, 
which is quite different from our setting. 
Data compression for  graphs
with a fixed alphabet is considered in \cite{MR1984489}.
In a different direction,  the case of {\em sequences} of length $N$ with increasing-sized alphabets 
is considered in \cite{MR2095850,MR2817014}.  
Closest to our topic is \cite{SZP}, discussing entropy and 
explicit compression algorithms for \ER\ random graphs.
But all of this literature deals with settings that seem  ``more mathematical" than ours, in the sense 
of being  less closely related  to compression of 
real-world graphical structures involving extra information.

On the applied side, there is considerable discussion of 
heuristic compression algorithms designed to exploit expected features of graphs arising in particular contexts, for instance 
WWW links 
\cite{Boldi03thewebgraph}
and social networks
\cite{Chierichetti:2009}.
What we proposed in the previous section as future research is to try to bridge the gap between that work and mathematical theory by seeking to devise and study general purpose heuristic algorithms.

On a more speculative note, we have a lot of sympathy with the view 
expressed by John Doyle and co-authors
\cite{doyle-anderson}, 
who argue that the ``organized complexity" one sees in real world 
evolved biological and technological networks is essentially different from the ``disorganized complexity" produced by probability models of random graphs.  
At first sight it is unclear how one  might try to demonstrate this distinction at some statistical level.  But producing a heuristic algorithm that codes some class of real-world networks to lengths 
smaller than the entropy of typical  probability models of such networks would be rather convincing.

\section{A little technical background}
\label{sec:backg}
Here are some elementary facts \cite{MR2239987} about entropy, 
in the setting (\ref{ent-def}) of a $S$-valued r.v. $X$, 
which we will use without comment.
\begin{eqnarray*}
\ent(X) & \le& \log |S| \\
\ent(X,Y) &\le & \ent(X) + \ent(Y)\\
\ent(X) & \ge & \ent(h(X)) \mbox{ for any } h:S \to S^\prime.
\end{eqnarray*}
These inequalities are equalities if and only if, respectively,

$X$ has uniform distribution on $S$

$X$ and $Y$ are independent 

$h$ is $1 - 1$ on the range of $X$.

\noindent
Also, if $\bar{\theta} = \sum_s q_s \theta_s$, 
where $q = (q_s)$ and each $\theta_s$ is a probability distribution, then 
\begin{equation}
 \ent(\bar{\theta}) \le \ent(q) + \sum_s q_s \ \ent(\theta_s) 
\label{H-mixing}
\end{equation}
with equality if and only if the supports of the $\theta_s$ are essentially disjoint.  In random variable notation, 
\begin{equation}
\ent(X) = \ent (f(X)) + \Ex \ent(X \vert f(X))
\label{H-condit}
\end{equation} 
where the random variable $\ent(X \vert Y)$ denotes entropy of the conditional distribution.  
(Note this is what a probabilist would call ``conditional entropy", though information theorists use that phrase to mean $\Ex \, \ent(X \vert Y)$). 
Write 
\[ \EE(p) = - p \log p - (1-p) \log (1-p) \]
for the entropy of the Bernoulli($p$) distribution.  
We will often use the fact 
\begin{equation}
 \EE(p) \sim  p \log \sfrac{1}{p} \mbox{ as } p \downarrow 0. 
\label{Bplim}
\end{equation}
We will also often use the following three basic crude estimates.
First,
\begin{equation}
\mbox{ if } K_m \to \infty \mbox{ and } \sfrac{K_m}{m} \to 0 
\mbox{ then } \log {m \choose K_m} \sim K_m \log \sfrac{m}{K_m} .
\label{mK}
\end{equation}
Second, for $X(n,p)$ with Binomial$(n,p)$ distribution, 
if $0\leq x_n \leq np$ and $x_n/n \to x \in [0,p]$ then
\begin{equation}
\log \Pr(X(n,p) \le x_n) = -n \Lambda_p(x_n/n) + \textrm{O}(\log n ), \label{BinLD}
\end{equation}
where $\Lambda_p(x) := x \log \sfrac{x}{p} + (1-x) \log \sfrac{1-x}{1-p}$.
The first order term is standard from large deviation theory and the second order
estimate follows from finer but still easy analysis; see for example 
Lemma~2.1 of~\cite{kamc10}.
Third, write $G[N,M]$ for the number of graphs on 
vertex-set $1,\ldots,N$ with at most $M$ edges.
It easily follows from~\eqref{BinLD} that
\begin{equation}
\mbox{if } \sfrac{M}{N} \to \zeta \in [0,\infty)  \mbox{ then }
 \frac{\log G[N,M]}{N \log N} \to \zeta. \label{GMN}
\end{equation}

\section{Easy examples}
\label{sec:easy}
Standard models of random graphs on vertices 
labelled $1, \ldots, N$ can be adapted to our setting of vertex-names in several ways.  In particular, one could either\\
(i) re-write the integer label in binary, that is as a binary string; or\\
(ii) replace the labels by distinct random strings as names.\\
These two schemes are illustrated in the first two examples below.

We present the results in a fixed format: 
a name for the model as a subsection heading, a definition of the model $\GG_N$, 
typically involving parameters $\alpha, \beta, \ldots$, and a Proposition giving a formula for the entropy rate $c = c(\alpha, \beta, \ldots)$ such that
\[  \ent(\GG_N) \sim c N \log N \mbox{ as } N \to \infty .\]
Model descriptions and calculations sometimes implicitly assume $N$ is sufficiently large.

These particular models are ``easy" in the specific sense that independence of edges allows us to write down an exact expression for entropy; then calculations establish the asymptotics.  
We also give two general results, Lemmas~\ref{L1} and~\ref{L2},
showing that graphs with short edges, or with  
similar names between connected vertices, have entropy rate zero. 

\subsection{Sparse \ER, default binary names}

{\bf Model.} $N$ vertices, whose names are the integers 
$1,\ldots, N$ written as binary strings of length $\lceil \log_2 N \rceil$.
Each of the ${N \choose 2}$ possible edges is present independently with probability $\alpha/N$, where $0<\alpha<\infty$.

\smallskip \noindent
{\bf Entropy rate formula.}
$c(\alpha) = \sfrac{\alpha}{2}$.

\smallskip  \noindent \proof
The entropy equals 
${N \choose 2} \EE(\alpha/N)$; 
letting $N \to \infty$ and using (\ref{Bplim}) gives the formula.

\subsection{Sparse \ER, random $A$-ary names}
\label{sec:SpERb}

{\bf Model.} As above, $N$ vertices, and
each of the ${N \choose 2}$ possible edges is present independently with probability $\alpha/N$.  Take $L_N \sim \beta \log_A N$ for 
$1 <  \beta < \infty$ and take the vertex names 
as a uniform random choice of $N$ distinct $A$-ary strings of length $L_N$.

\smallskip \noindent
{\bf Entropy rate formula.}
$c(\alpha, \beta) = \beta - 1 + \frac{\alpha}{2}$.

\smallskip  \noindent \proof
The entropy equals 
$\log{A^{L_N} \choose N} +{N \choose 2} \EE(\alpha/N)$. 
The first term $\sim (\beta - 1) N \log N$ by (\ref{mK}) and the second 
term $\sim \frac{\alpha}{2} N \log N$ as in the previous model.

\smallskip \noindent
{\bf Remark.} One might have naively guessed that the formula would involve $\beta$ 
instead of $\beta  -1$, on the grounds that the entropy of the sequence of names is $\sim \beta N \log N$, but this is the rate in a third model where 
a vertex name is a pair $(i,\ba)$, where $1 \le i \le N$ and $\ba$ is the random string.  
This model distinction  becomes more substantial for the model to be studied
in Section~\ref{sec:hard}.

\subsection{Small Worlds Random Graph}
\label{sec:SW} 

{\bf Model.}
Start with $N=n^2$ vertices arranged in an $n\times n$ discrete torus,  
where the name of each vertex is its coordinate-pair $(i,j)$ written 
as two binary strings of lengths  $\lceil \log_2 n \rceil$.  
Add the usual edges of the degree-$4$ nearest neighbor torus graph.
 Fix parameters 
$0 < \alpha, \gamma < \infty$. 
For each edge $(w,v)$ of the remaining set $S$ of ${N \choose 2}-2N$ possible edges in the graph, add the edge independently 
with probability $p_N(||w-v||_2)$, where $p_N(r)=a r^{-\gamma}$ and $a:=a_{N,\gamma}$ is chosen such that the mean degree of the graph 
$\GG_N$ of these random edges $\to \alpha$ as $N \to \infty$ 
(see (\ref{sw-d1},\ref{sw-d2}) for explicit expressions)
and the Euclidean distance $||w-v||_2$ is taken using the torus convention. 

\smallskip \noindent
{\bf Entropy rate formula.}
\begin{eqnarray*}
c(\alpha, \gamma) &=& \alpha/2, \quad 0 < \gamma < 2 \\
&=& \alpha/4, \quad \gamma = 2 \\
&=& 0,  \quad \quad 2< \gamma < \infty .
\end{eqnarray*}

\smallskip \noindent
{\bf Remark.} The different cases arise because for $\gamma < 2$ the 
edge-lengths are order $n$ whereas for $\gamma > 2$ they are $O(1)$.

\smallskip  \noindent \proof 
Write $r_{i,j}=\sqrt{i^2+j^2}$ and $p_{i,j}=p_N(r_{i,j})$.  The degree $D(v)$  of vertex $v$ in $\GG_N$ (assuming $n$ is odd -- 
the even case is only a minor modification) has mean
\begin{align} 
\Ex D(v)-4 &= 4\sum_{i,j=1}^{(n-1)/2} p_N(r_{i,j})+ 4\sum_{i=2}^{(n-1)/2} p_N(r_{i,0}) \notag \\
	&=    a\left( 4 \sum_{i,j=1}^{(n-1)/2} (i^2+j^2)^{-\gamma/2}+ 4 \sum_{i=2}^{(n-1)/2} i^{-\gamma}\right)      .             \label{sw1}
\end{align}
Similarly, the entropy  
of $\GG_N$ is exactly 
\begin{equation}
 \label{2}
\ent(\GG_N)=\frac{N}{2}\left(4\sum_{i,j=1}^{(n-1)/2} \EE(p_N(r_{i,j}))+ 4\sum_{i=2}^{(n-1)/2} \EE(p_N(r_{i,0}))\right).
\end{equation}
One can analyze these expressions separately in the three cases.
First consider the ``critical" case  $\gamma = 2$.  Here the quantity in parentheses in (\ref{sw1}) 
is $\sim \int_1^{(n-1)/2} 2\pi r^{-1} dr \sim 2 \pi  \log n \sim \pi\log N$.  
We therefore take 
\begin{equation}
a = a_{N,1} \sim \sfrac{\alpha}{\pi \log N}
\label{sw-d1}
\end{equation}
so that $\GG_N$ has mean degree  $\to 4+\alpha$. 
Evaluating the entropy similarly, 
where in the second line the ``$\log a$" term is asymptotically negligible,
\begin{eqnarray*}
\ent(\GG_N) &\sim& \frac{N}{2} 
\int_1^{(n-1)/2} 2\pi r \  \EE(ar^{-2}) dr \\
 &\sim& N \pi  
\int_1^{(n-1)/2} r  \cdot ar^{-2} \cdot ( - \log a + 2 \log r ) \ dr \\
 &\sim& 2N \pi  a
\int_1^{(n-1)/2} r^{-1} \log r  \ dr \\
 &\sim& 2N \pi  a \cdot \sfrac{1}{2} \log^2 n \\
&\sim &\sfrac{\alpha}{4} N \log N
\end{eqnarray*} 
giving the asserted entropy rate formula in this case $\gamma = 2$.

In the case $\gamma < 2$, 
more elaborate though straightforward calculations 
(see appendix)
show that to have the mean degree $\to \alpha+4$ we take
\begin{equation}
a = a_{N,\gamma} \sim \alpha \kappa_\gamma  N^{-1 + \gamma/2} ; 
\quad \kappa_\gamma = \frac{2 - \gamma}{2^{1+\gamma} 
\int_0^{\pi/4} \sec^{2-\gamma}(\theta) d\theta} . 
\label{sw-d2}
\end{equation} 
and then establish the asserted entropy rate $\alpha/2$.

In the case $\gamma > 2$ the mean length of the edges of $\GG_N$ 
becomes $O(1)$. 
One could repeat calculations for this case, but the asserted zero 
entropy rate follows from the more general Lemma \ref{L1} later, 
as explained in Section~\ref{sec:zero_rate}.

\smallskip \noindent
{\bf Remark.} 
The case $\gamma < 2$ suggests a general principle that
models with  ``long edges" should have the same entropy rates as if 
the edges were uniform random subject to the same degree distribution. 
But there seems no general  formulation 
of such a result without explicit dependence assumption.

\subsection{Edge-probabilities depending on Hamming distance}

We first describe a general model, then the specialization that
we shall analyze.

{\bf General model.} 
Fix an alphabet $\bA$ of size $A$.
For each $N$ choose $L_N$ such that $N \le A^{L_N}$, 
and suppose $L_N \sim \beta \frac{\log N}{\log A}$ for some 
$\beta \in [1,\infty)$.
Take $N$ vertex-names as a uniform random choice of distinct length-$L_N$ strings from $\bA$. 
Write $d_H(\ba,\ba^\prime) = |\{i: a_i \neq a^\prime_i\}|$ for Hamming distance between names.  
For each $N$ let $\bw = \bw^N$ be a sequence of decreasing weights 
$1 = w(1) \ge w(2) \ge \ldots \ge w(L_N) \ge 0$.
We want the probability of an edge between vertices 
$(\ba,\ba^\prime)$ to be proportional to $w(d_H(\ba,\ba^\prime))$.
For each vertex $\ba$, the expectation of the sum of 
$w(d_H(\ba,\ba^\prime))$ over other vertices $\ba^\prime$ equals
\begin{equation}
 \mu_N:=
\frac{N-1}{1 - A^{-L_N}} 
\sum_{u=1}^{L_N} {L_N \choose u} \left(\frac{A-1}{A}\right)^u \left(\frac{1}{A}\right)^{L_N-u} \ w(u). \label{Hamm-1}
\end{equation}
Fix $0 < \alpha < \infty$, and make a random graph $\GG_N$ with mean degree $\alpha$ by specifying that, conditional on the set of vertex-names, each possible edge 
$(\ba,\ba^\prime)$ is present independently with probability 
$\alpha w(d_H(\ba,\ba^\prime)) / \mu_N$.

Note that in order for this model to make sense, we need $\mu_N\geq\alpha$, which is not guaranteed
by the description of the model. 

Intuitively, we expect  that the lengths (measured by Hamming distance) of edges
will be around the $\ell_N$ maximizing $\binom{L_N}{\ell_N} \ w^N(\ell_N) $, and that for all (suitably regular) choices of $\bw^N$ with 
$\ell_N/N \to d \in [0,1]$ the entropy rate will involve $\bw^N$ only via the limit $d$. 
Stating and proving a general such result seems messy, so
we will study only the special case 
\begin{eqnarray}
w(u)&=&1, \  1\leq u\leq M_N; \label{wu} \\
  &=& 0 ,\  M_N <u \le L_n \nonumber \\ 
M_N/L_N &\to & d\in(0,1 - \sfrac{1}{A})  \\
1\leq &\beta&<  \frac{\log A}{\Lambda_{1 - 1/A}(d) } \label{betaUB}
\end{eqnarray}
for $\Lambda_p(d)$ as at (\ref{BinLD}).  
Here condition (\ref{betaUB}) is needed, as we will see at (\ref{munb}), to 
make $\mu_N \to \infty$.  
Note that for the case $d = 0$ one could use Lemma \ref{L1} later and the accompanying conditioning argument in Section~\ref{sec:zero_rate} to show that the entropy rate of $\GG_N$ equals the rate ($\beta - 1$) for the 
set of vertex-names.  
The opposite case 
 $(1 - 1/A) \le d \le 1$ is essentially the model of Section~\ref{sec:SpERb} 
and the rate becomes $\beta - 1 + \alpha/2$: as expected, these rates are the $d \to 0$ and the $d \to  1 - 1/A$ limits of the rates in our formula below.

\smallskip \noindent
{\bf Entropy rate formula.}  
In the special case (\ref{wu} - \ref{betaUB}),
\[
c(A, \alpha, \beta, d)=\beta-1+\frac{\alpha}{2}\left(1-\frac{\beta\Lambda_{1-1/A}(d)}{\log A}\right) .
\]
To establish this formula, first observe that
 for Binomial $X(\cdot,\cdot)$ as at (\ref{BinLD})
\baln{
\mu_N=\frac{N-1}{1 - A^{-L_N}} \Pr(1\leq X(L_N,1 - 1/A) \leq M_N), \label{bdt2}
}
and so by  (\ref{BinLD}) 
\begin{equation}
 \frac{\log \mu_N}{\log N} \to 1 -  \frac{ \beta \ \Lambda_{1-1/A}(d) }{\log A} .
\label{munb}
\end{equation}
So condition (\ref{betaUB}) ensures that $\mu_N \to \infty$ and therefore the model makes sense.
Write $\Names$ (to avoid overburdening the reader with symbols) for the random unordered set of vertex-names, 
and use (\ref{H-condit}) to write
\[ \ent(\GG_N) = \ent(\bL) + \Ex \, \ent(\GG_N \vert \bL) . \]
As in Section~\ref{sec:SpERb} the contribution to the entropy rate from the first term is 
$\beta - 1$.  For the second term, write
\[ \ent(\GG_N \vert \bL) = \sum_{\ba\not=\ba^\prime} 
 \EE\left(\sfrac{\alpha}{\mu_N}\right)  \Ind_{(\ba \in \bL, \ba^\prime \in \bL)} 
\Ind_{(d_H(\ba,\ba^\prime) \le M_N)} 
\]
where the sum is over unordered pairs $\{\ba,\ba^\prime\}$ in $\bA^{L_N}$.
Take expectation to get 
\[ \Ex \, \ent(\GG_N \vert \bL) =
\frac{\binom{N}{2}}{1 - A^{-L_N}} 
\sum_{u=1}^{M_N} {L_N \choose u} \left(\frac{A-1}{A}\right)^u \left(\frac{1}{A}\right)^{L_N-u} \EE\left(\frac{\alpha}{\mu_N}\right) .
\]
But from the definition (\ref{Hamm-1}) of $\mu_N$ this simplifies to 
$\frac{N}{2} \mu_N \EE(\alpha/\mu_N)$, 
and then from (\ref{Bplim})
\[ \Ex \, \ent(\GG_N \vert \bL)  \sim \frac{\alpha N}{2}\ \log \mu_N.  \]
Appealing to (\ref{munb}) establishes the entropy rate formula.

\subsection{Non-uniform and uniform random trees}
\label{sec:tree}
{\bf Model.}
Construct a random tree $\TT_N$ on vertices $1,\ldots, N$ as follows.  
Take $V_3, V_4, \ldots, V_N$ independent uniform on 
$\{1,\ldots,N\}$.  Link vertex $2$ to vertex $1$.  
For $k = 3,4,\ldots, N$ link vertex  $k$ to vertex $\min(k-1, V_k)$.

\smallskip \noindent
{\bf Entropy rate formula.}
$c = 1/2$.

\smallskip  \noindent \proof 
\[ \ent(\TT_N) = \sum_{k=3}^N \ent(W_k) \]
where $W_k = \min(k-1, V_k)$ has entropy 
\[ \ent(W_k) = \sfrac{k-2}{N} \log N + \sfrac{N-k+2}{N} \log \sfrac{N}{N-k+2} \]
The sum of the first term $\sim \frac{1}{2} N \log N$ and the sum of the second term is of smaller order.

\smallskip \noindent
{\bf Remark.}  This tree arose in \cite{MR1085326}, where it was shown (by an indirect argument)  that
if one first constructs $\TT_N$, then applies a uniform random permutation to the vertex-labels, 
 the resulting  random tree $\TT^*_N$ is uniform on the set of all labelled trees.  Cayley's formula tells us there are $N^{N-2}$ labelled trees, so 
$\ent(\TT^*_N) = \log N^{N-2}$and so 
 $(\TT^*_N)$ has entropy rate $c = 1$.

\subsection{Conditions for zero entropy rate}
\label{sec:zero_rate}
Here we will give two complementary conditions under which the 
entropy rate is zero.  
Lemma \ref{L1} concerns the case where we start with deterministic vertex-names, and add random edges which mostly link a vertex to some of the 
``closest" vertices, specifically to vertices amongst the
($o(N^\eps)$ for all $\eps > 0$) closest vertices. 
Lemma \ref{L2} concerns the case where we start with a determinstic graph 
on unlabelled vertices, and add random vertex-labels such that vertices linked by an edge mostly have names that differ in only $o(\log N)$ places.
Note that these lemmas may then be applied conditionally.  
That is, if we start with a random unordered set of names, and then 
(conditional on the set of names) add random edges in a way satisfying the 
assumptions of Lemma \ref{L1},
then the entropy rate of the resulting $\GG_N$  will equal the entropy rate of the 
original random unordered set of names.
Similarly,  if we start with a random graph 
on unlabelled vertices, then (conditional on the graph) add random names 
 in a way satisfying the 
assumptions of Lemma \ref{L2},
then the entropy rate of the resulting $\GG_N$  will equal the entropy rate of the original random unlabelled graph.

\begin{Lemma}\label{L1}
For each $N$, suppose we take $N$ vertices with deterministic names 
(w.l.o.g.  $1 \le i \le N$ written as binary strings, to fit our set-up) and suppose for each $i$ we are given an ordering 
$j(i,1), j(i,2), \ldots, j(i,N-1)$ of the other vertices. 
Say that an edge $(i, j = j(i,\ell))$ with $i<j$ has length $\ell$.
Consider a sequence of random graphs $\GG_N$ whose distribution is arbitrary subject to \\
(i) The number $E_N$ of edges satisfies 
$\Pr(E_N>N\beta)=\textrm{o}(N^{-1} \log N)$ for some constant $\beta < \infty$;\\
(ii) For some $M_N$  such that 
$\log(M_N)=\textrm{o}(\log N)$, the r.v. \\
\hspace*{0.5in} $X_N := \mbox{
number of edges with length greater than $M_N$} $ \\
satisfies $\Pr( X_N > N\delta)= \textrm{o}(1)$ for all $\delta>0$.\\
Then the entropy rate is $ c = 0$.
\end{Lemma}
\smallskip \noindent
{\bf Remark.} The lemma applies to the $\gamma>2$ case of the ``small worlds" model 
in Section~\ref{sec:SW}.  Take the ordering  induced by the
natural distance between vertices.  In this case, $E_N$ is a sum of independent indicators
with $\Ex E_N \sim c N$ for some constant $c$. Standard concentration results (e.g.~\cite{MR2248695} Theorem 2.15)
imply (i) for any $\beta>c$, and (ii) follows since for any sequence
$M_N\to\infty$ we have $\Ex X_N= \textrm{O}(N M_N^{1-\gamma/2})=\textrm{o}(N)$.

\smallskip  \noindent \proof
We first show that the result holds with (i) replaced by \\
(i$'$)  $\Pr(E_N>N\beta)=0$,\\
and then use this modified statement to prove the lemma.

Assume now that $\GG_N$ satisfies (i$'$) and (ii) and write
$\GG_N$, considered as an edge-set, as a disjoint union 
$\GG_N^\prime \cup \GG_N^{\prime \prime}$, 
where $\GG_N^\prime$ consists of the edges of length $\le M_N$.  
Because $\GG_N^\prime$ contains at most $\beta N$ edges out of a set of 
at most $NM_N$ edges, 
\begin{eqnarray*}
 \ent (\GG_N^\prime) &\le& \log(\beta N) + \log {NM_N \choose \beta N} \\
&=& o(N \log N) \mbox{ by } (\ref{mK}).
\end{eqnarray*}
Now fix $\delta > 0$ and condition on  whether the number $X_N$ of edges of $\GG_N^{\prime \prime}$ is bigger or smaller than $\delta N$.  
Using (\ref{H-mixing}) we get 
\[
 \ent(\GG_N^{\prime \prime}) 
\le \log 2 + \log G[N,\delta N] + \Pr(X_N > \delta N) \log G[N,\beta N]  . \]
Now $\Pr(X_N > \delta N) \to 0$ by assumption (ii), and then using (\ref{GMN}) we get
\[  \ent(\GG_N^{\prime \prime})  \le (\delta + o(1)) N \log N. \]
Because $\delta > 0 $ is arbitrary we conclude 
\[ \ent(\GG_N) \le \ent (\GG_N^\prime)  +  \ent(\GG_N^{\prime \prime}) 
= o(N \log N) . \]

Now assume that $\GG_N$ satisfies the weaker hypotheses (i) and (ii). 
Defining $\widehat{\GG_N}$ to have the conditional distribution of 
$\GG_N $ given $E_N\leq \beta N$,
it is clear that 
$\widehat{\GG_N}$ satisfies (i$'$).
We will show that it also satisfies (ii), implying (by the previous result)  that the entropy rate of $\widehat{\GG}_N$ is zero.  
Let $\delta>0$.  Conditioning on  the event ($A_N$, say) that $E_N\leq \beta$,
\baln{
\Pr(X_N>\delta N) = \Pr( X_N>\delta N| A_N)\Pr(A_N) + \Pr(X_N>\delta N|A_N^c)\Pr(A_N^c). \label{t5}
}
By (ii), the term on the left hand side of \eqref{t5} is $\textrm{o}(1)$, and 
by (i), $\Pr(A_N^c)=\textrm{o}(1)$, and so also $\Pr(A_N) \to 1$.  Thus,
$\Pr( X_N>\delta N| A_N)$ must be $\textrm{o}(1)$, as desired.

To complete the proof, use \eqref{H-mixing} to write
\bal{
\ent(\GG_N)&\leq \EE(\Pr(A_N))+ \ent(\GG_N|A_N)\Pr(A_N)+\ent(\GG_N|A_N^c)\Pr(A_N^c) \\
	&\leq \log 2+\ent(\widehat{\GG}_N)+ \Pr(A_N^c) \binom{N}{2}\log 2 .
}
The entropy rate of $\widehat{\GG}_N$ is zero, and assumption (i) is
exactly that $\Pr(A_N^c)=\textrm{o}(N^{-1} \log N)$, so $\ent(\GG_N) = 
\textrm{o}(N\log N)$, as desired.

\begin{Lemma}
\label{L2}
Take a deterministic graph on $N$ unlabelled vertices, and let $c_N$ 
denote the number of components and $e_N$ the number of edges.  Construct $\GG_N$ by assigning random distinct vertex-names $\ba(v)$
of length $\textrm{O}(\log N)$ to vertices $v$,  their distribution being arbitrary subject to 
\[ \sum_{\textrm{edges } (\ba, \ba^\prime)} d_H(\ba, \ba^\prime) = \textrm{o}(N \log N) 
\mbox{ in probability}. \]
If $e_N = O(N)$ and $c_N = o(N)$ then $\GG_N$ has 
entropy rate zero.
\end{Lemma}

\smallskip  \noindent \proof 
By a straightforward truncation argument 
we may assume there is a deterministic bound 
\[ \sum_{\textrm{edges } (\ba, \ba^\prime)} d_H(\ba, \ba^\prime) 
\le s_N = \textrm{o}(N \log N) .\] 
The name-lengths are $\le \beta \log N$ for some $\beta$.
Consider first the case where there is a single component. 
Take an arbitrary spanning tree with arbitrary root, and write the edges of the tree in breadth-first order as $e_1,\ldots, e_{N-1}$.  
We can specify $\GG_N$ by specifying first
the name of the root; then 
for each edge $e_i = (v,v^\prime)$ directed away from the root, 
specify the coordinates where $\ba(v^\prime)$ differs from $\ba(v)$ 
and specify the values of $\ba(v^\prime)$ at those coordinates.
Write $\SS$ for the random set of  all these differing coordinates.
Conditional on $\SS = S$ the entropy of $\GG_N$ is at most
$(|S| + \beta \log N ) \log A$, where the $\beta \log N$ term 
arises from the root name.
So using  \eqref{H-mixing} 
\[ \ent(\GG_N) \le \ent(\SS) + (s_N + \beta \log N) \log A . \]
With $c_N$ components the same argument shows 
\[ \ent(\GG_N) \le \ent(\SS) + (s_N + c_N \beta \log N) \log A . \]
The second term is $\textrm{o}(N \log N)$ by assumption, and 
\[ \ent(\SS) \le \log \left (\sum_{i \le s_N} 
\binom{\beta N \log N}{i} \ \right) = \textrm{o}(N \log N) , \]
the final relation by e.g.\ the $p = 1/2$ case of (\ref{BinLD}).

\subsection{Summary}
The reader will recognize the models in this section as standard random graph models, 
adapted to our setting in one of several ways.
One can take a model of dynamic growth, adding one vertex at a time, and then assign the $k$'th vertex a name, e.g.\ the 
``default binary" or ``random A-ary" used in the \ER\ models. 
Alternatively, as in the ``Hamming distance" model, one can
start with $N$ vertices with assigned names and then add edges according to some probabilistic rule involving the names of end-vertices.
Roughly speaking, for any existing random graph model where one can calculate anything, one can calculate the entropy rate for such adapted models. 
But this is an activity perhaps best left for future Ph.D. theses.
We are more interested in models where the graph structure and the name structure each simultaneously influence the other, rather than starting by specifying one structure 
and having that influence the other.  
It is not so easy to devise tractable such models, but the next section
shows our attempt.

\section{A hybrid model}
\label{sec:hard}
In this section we study a model for which calculation of the entropy rate is less straightforward.  
It incidently reveals a connection between our setting and the more familiar setting of ``graph entropy".

\subsection{The model}
\label{sec:hybrid_model}
In outline, the graph structure is again sparse \ER\  
$\GG(N,\alpha/N)$, but we construct it inductively over vertices, and make the vertex-names copy parts of the names of previous vertices that the current vertex is linked to.  
Here are the details.

\smallskip
\noindent
{\bf Model: \ER\ with hybrid names.} Take $L_N \sim \beta \log_A N$ for $1<\beta<\infty$.  
Vertex $1$ is given a uniform random length-$L_N$ $A$-ary name.
For $1 \le n \le N-1$:
\begin{quote}
vertex $n+1$ is given an edge to each vertex $ i \le n$ independently with probability $\alpha/N$.  Write $Q_n \ge 0$ for the number of such edges, 
and $\ba^1, \ldots, \ba^{Q_n}$ for the names of the linked vertices. 
Take an independent uniform random length-$L_N$ $A$-ary string $\ba^0$.
Assign to vertex $n+1$ the name obtained by, independently for each 
coordinate $1 \le u \le L_N$, making a uniform random choice from the 
$Q_n+1$ letters 
$a^0_u, a^1_u, \ldots,a^{Q_n}_u$.
\end{quote}
See Figure 1. 
This model gives a family $(\GG_N)$ parametrized by $(A,\beta,\alpha)$. 
Note that this scheme for defining ``hybrid" names could be used
with any sequential construction of a random graph, for instance 
preferential attachment models.

\setlength{\unitlength}{0.33in}
\begin{picture}(10,4)(-3,0)
\put(0,1){\circle*{0.2}}
\put(2,3){\circle*{0.2}}
\put(5,2){\circle*{0.2}}
\put(-0.6,1.26){d\underline{af}cbb}
\put(1.4,3.26){bfc\underline{ca}d} 
\put(4.6,2.26){bafcac}
\put(0,1){\vector(5,1){4.7}}
\put(2,3){\vector(3,-1){2.7}}
\end{picture}

{\bf Figure 1.} 
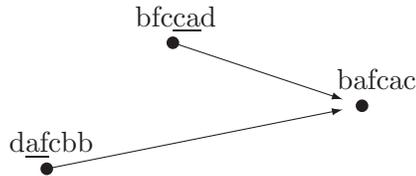
{\small Schematic for the hybrid model.  
A vertex (right) arrives with some ``original name" 
$\underline{b}bdab\underline{c}$ and is attached to two previous vertices with names 
$dafcbb$ and $bfccad$.  
The name given to the new vertex is obtained by copying for each 
position the letter in that position in a uniform random choice from the three names.  Choosing the underlined letters gives the name shown in the figure.}

\subsection{The ordered case}

This model illustrates a distinction mentioned in
Section~\ref{sec:SpERb}.  
In the construction above, the  $n$th vertex is assigned a name, say $\ba^n$, during the construction, but in the final graph $\GG_N$ we do not see the value of $n$ for a vertex.
The ``ordered" model $(\GG^{ord}_N)$ in which we do see the value of $n$ for each vertex, by making the name be 
$(n,\ba^n)$, is a different model whose analysis is conceptually more straightforward, so we will start with that model.  
We return to the unordered model in section \ref{sec-unordered}.

\smallskip \noindent
{\bf Entropy rate formula} for $(\GG^{ord}_N)$.
\begin{equation}
\frac{\alpha}{2} +  \beta \sum_{k \ge 0}  \frac{\alpha^k J_k(\alpha) h_A(k) }{k! \log A } 
\label{hybrid-rate}
\end{equation}
where 
\[ J_k(\alpha) := \int_0^1 x^k e^{-\alpha x} dx \]
and the constants $h_A(k)$ are defined at (\ref{hA-def}).

Write $\GG_{N,n}$ for the partial graph obtained after vertex $n$ has been assigned its edges to previous vertices and then its name.
We will show  that, for deterministic $e_{N,n}$ defined at (\ref{eNn}) below, 
as $N \to \infty$  
the entropies of the conditional distributions satisfy
\begin{equation}
\max_{1 \le n \le N-1}  
\Ex \left| \ent(\GG_{N,n+1} \vert \GG_{N,n}) - e_{N,n} \right| 
= \textrm{o}(\log N) .
\label{eNn1}
\end{equation}
By the chain rule  (\ref{H-condit}) this immediately implies 
\[ \ent(\GG_N^{ord}) - \sum_{n=1}^{N-1} e_{N,n} = \textrm{o}(N \log N) \]
which will establish the entropy rate formula.

The key ingredient is the following technical lemma; 
note that the measures $\mu^{\textbf{i}}$ below depend on the 
realization of $\GG^{ord}_N$ and are therefore random quantities.    
Write ``ave" for average, and write 
\[ ||\Theta||\ta{k} : = \sfrac{1}{2}  \sum_{\ba \in \bA^k} 
\left| \Theta(\ba) - A^{-k} \right|  \] 
for the variation distance between a probability distribution $\Theta$ 
on $\bA^k$ and the uniform distribution.
\begin{Lemma}
\label{L3a}
Write $(n,\ba^n), \ 1 \le n \le N$ for the vertex-names of $\GG^{ord}_N$.
For each $k\geq 1$ and $\textbf{i}:=(i_1,\ldots,i_k)$ with $1 \le i_1<\cdots < i_k \le N$,
write $\mu^{\textbf{i}}$ 
for the empirical distribution of 
$(a^{i_1}_u,\ldots, a^{i_k}_u), 1 \le u \le L_N$.  
That is, the probability distribution on $ \bA^k$ 
\[ \mu^{\textbf{i}}(x_1,\ldots, x_k) := L_N^{-1} \sum_{u=1}^{L_N} 
\Ind_{\left(a^{i_1}_u = x_1,\ldots, a^{i_k}_u = x_k\right)}. \]
Then
\baln{
\Delta\ta{k}_N := 
\max_{2 \le n \le N} 
\Ex ||\underset{1 \le i_1<\cdots < i_k \le n}{\ave} \ \mu^{\textbf{i}} ||\ta{k}\leq C \left(\frac{A^{k/2}}{\sqrt{\log N}}+\frac{k^2}{N}\right).\label{089}
}
for a constant $C$ not depending on $k, N$.
\end{Lemma}
We defer the proof to Section \ref{sec:p3}.

Fix $N$ and $n$, and consider 
$\ent(\GG_{N,n+1} \vert \GG_{N,n})$, 
the entropy of the conditional distribution.   
Conditioning on the 
edges of vertex $n+1$ in $\GG_{N,n+1}$, and using the chain rule~\eqref{H-condit}, we find
\baln{
&\ent(\GG_{N,n+1} \vert \GG_{N,n})=n \EE(\alpha/N) \notag \\
& \quad +\hspace{-5mm}\mathop{\sum_{k=0,\ldots, n}}_{1\leq i_1<\cdots<i_k\leq n}\hspace{-5mm}\left(\frac{\alpha}{N}\right)^k\left(1-\frac{\alpha}{N}\right)^{n-k}
\ent(\ba^{n+1}\vert \GG_{N,n}, n+1\rightarrow \{i_1,\ldots,i_k\}), \label{155}
}
where $n+1\rightarrow\{i_1,\ldots,i_k\}$ denotes the event that vertex $n+1$ connects
to vertices $i_1,\ldots,i_k$ and no others.
The contribution to the entropy from the choice of edges is
$n \EE(\alpha/N)$, which as in previous models contributes 
(after summing over $n$) the first term 
$\alpha/2$ of the entropy rate formula, so in the following we need
consider only the contribution from names, that is the sum in~\eqref{155}.  
Consider the contribution to the sum~\eqref{155} from $k=2$,
that is on the event 
$\{ Q_n = 2\}$ that vertex 
$n+1$ links to exactly two previous vertices. 
Conditional on these being a particular pair $1 \le i < j \le n$, 
with names $\ba^i, \ba^j$, the contribution to entropy is exactly
\[ \ent(\ba^{n+1}\vert \GG_{N,n}, n+1\rightarrow \{i,j\})=L_N \sum_{(a,a^\prime) \in \bA \times \bA} 
g_2(a,a^\prime) \ \mu^{(i,j)}(a,a^\prime) \] 
where 
\begin{eqnarray*}
 g_2(a,a^\prime)  &=& \EE_A(\sfrac{A+1}{3A}, \sfrac{A+1}{3A}, 
\sfrac{1}{3A}, \sfrac{1}{3A}, \ldots \ldots \sfrac{1}{3A}) 
\mbox{ if } a^\prime \neq a \\
 &=& \EE_A(\sfrac{2A+1}{3A}, \sfrac{1}{3A},
\sfrac{1}{3A}, \sfrac{1}{3A}, \ldots \ldots \sfrac{1}{3A}) 
\mbox{ if } a^\prime = a 
\end{eqnarray*}
and where $\EE_A(\bp)$ is the entropy of a distribution 
$\bp = (p_1,\ldots,p_A)$. 
Now unconditioning on the pair $(i,j)$, the contribution to 
$\ent(\GG_{N,n+1} \vert \GG_{N,n})$ from the event $\{ Q_n = 2\}$; that is
the $k=2$ term of the sum~\eqref{155}; equals
\[ L_N \sum_{1 \le i < j \le n} \frac{\alpha^2}{N^2} 
\ \left(1 - \sfrac{\alpha}{N}\right)^{n-2}
\quad 
\sum_{(a,a^\prime) \in \bA \times \bA} 
g_2(a,a^\prime) \ \mu^{(i,j)}(a,a^\prime) \] 
\begin{equation}
 =  \frac{L_N \alpha^2 \binom{n}{2}}{N^2} 
\ \left(1 - \sfrac{\alpha}{N}\right)^{n-2}
\sum_{(a,a^\prime) \in \bA \times \bA} 
g_2(a,a^\prime) \ \underset{1 \le i < j \le n}{\ave} \ \mu^{(i,j)}(a,a^\prime) .
\label{LNa-sum}
\end{equation}
Lemma~\ref{L3a} now tells us that the sum in (\ref{LNa-sum}) differs from 
\[ h_A(2) := A^{-2} \sum_{(a,a^\prime) \in \bA \times \bA} 
g_2(a,a^\prime) \] 
by at most 
$2g_2^* \Delta_N\ta{2} $ 
where $g_2^*\leq\log A$ is the maximum possible value of $g_2(\cdot, \cdot)$ and $\Delta_N\ta{2} $ is as defined in Lemma~\ref{L3a}.
So to first order as $N \to \infty$, the  
quantity (\ref{LNa-sum}) is
\[ e_{N,n,2} := \beta \log_A N \times \alpha^2 h_A(2) \sfrac{\binom{n}{2}}{N^2} 
 \ \exp(-\alpha n/N),\]
 with an error bounded by
 \bal{
 (2\log A) \, L_N\binom{n}{2}\left(\frac{\alpha}{N}\right)^2\left(1-\frac{\alpha}{N}\right)^{n-2}\Delta_N\ta{2}. 
 } 
A similar argument applies to
the terms in the sum~\eqref{155} for a general number $k$ of links.
In brief, we define
\[ e_{N,n,k} := \beta \log_A N \times \alpha^k h_A(k) \sfrac{\binom{n}{k}}{N^k} 
 \ \exp(-\alpha n/N) \] 
where
\begin{equation}
 h_A(k) := A^{-k} \sum_{(a_1,\ldots,a_k) \in \bA^k} 
\ent(\bp^{[a_1,\ldots,a_k]}),
\label{hA-def}
\end{equation}
and where $\bp^{[a_1,\ldots,a_k]}$ is the probability distribution $\bp$ on $\bA$ 
defined by
\[ p^{[a_1,\ldots,a_k]}(a) = \frac{1 + A\times |\{i: a_i = a\}|}{(1+k)A} . \] 
Also for $k = 0$ we set $h_A(0) = \log A$, the entropy of the uniform distribution on~$\bA$.   
Repeating the argument from the case $k=2$, 
we find that~\eqref{155} is, to first order, $\sum_{k\geq0} e_{N,n,k}$, with error of order
\baln{
L_N\sum_{k=0}^n \binom{n}{k}\left(\frac{\alpha}{N}\right)^k\left(1-\frac{\alpha}{N}\right)^{n-k}\Delta_N\ta{k}. \label{156}
}
Applying
Lemma~\ref{L3a} to bound $\Delta_N\ta{k}$
and then using simple properties of the binomial 
distribution yields that~\eqref{156} is $\textrm{o}(\log N)$. 

So we are now in the setting of~(\ref{eNn1}) with 
\begin{equation}
e_{N,n} = \sum_{k \ge 0} e_{N,n,k} .
\label{eNn}
\end{equation}  
Because 
\[ \sum_{n=1}^{N-1}  \sfrac{\binom{n}{k}}{N^k} 
 \ \exp(-\alpha n/N)
\sim \sfrac{N}{k!}  J_k(\alpha)\] 
calculating $\sum_{n=1}^{N-1} e_{N,n}$ gives the stated entropy rate formula.

\subsection{Proof of Lemma \ref{L3a}}
\label{sec:p3}
Fix $N$.   
Recall the construction of $\GG^{ord}_N$ involves an ``original name process" --  letters
of the name of vertex $n$ may be copies from previous names or may be from an ``original name", independent uniform 
for different $n$. 
Consider a single coordinate, w.l.o.g. coordinate $1$, of the 
vertex-names of $\GG^{ord}_N$.  
For each vertex $n$ this is either from the original name of $n$ or a copy of some 
previous vertex-name, so inductively the letter at vertex $n$ is a copy of the letter
originating at some vertex $1 \le C^N_1(n) \le n$;  
and similarly the letter at general coordinate $u$ is a copy from some vertex 
$C^N_u(n)$.
Because the copying process is independent of the name origination process, 
it is clear that the (unconditional) distribution of each name $\ba^n$ is
uniform on length-$L_N$ words.  
Moreover it is clear that, for $1 \le i < j \le N$, 
\begin{equation}\label{CNu}
\begin{split}
&\mbox{ the two names 
$\ba^i$ and $\ba^j$ are independent uniform} \\
&\hspace{1cm}\mbox{on the event }
\{C^N_u(i) \ne C^N_u(j) \ \forall u\}.
\end{split}
\end{equation}
The proof of Lemma~\ref{L3a} rests upon the 
following lemma, whose proof we defer
to the end of Section~\ref{sec:sparseER} . 
\begin{Lemma}
\label{L4}
For $(I,J)$ uniform on $\{1 \le i < j \le n\}$, write 
$\theta_{N,n} = \Pr (C^N_1(I) = C^N_1(J))$.  Then
\[ \max_{2 \le n \le N} \theta_{N,n} = \textrm{O}(1/N) 
\mbox{ as } N \to \infty .\]
\end{Lemma}

We first use this lemma to prove Lemma~\ref{L3a} in the case where $k=2$.
For $(I, J)$ as in Lemma~\ref{L4},
\begin{eqnarray}
\Delta_N\ta{2} &=&
 \sfrac{1}{2}\max_{2 \le n \le N} \Ex \sum_{\textbf{x}\in \textbf{A}^2}
 \left|
\Ex \big( L_N^{-1}\sum_{u=1}^{L_N} \Ind_{(a^{I}_u = x_1, a^{J}_u = x_2)} | \GG_{N,n} \big)
-A^{-2}  \right| \nonumber\\
&\leq& \sfrac{1}{2}\max_{2 \le n \le N} \sum_{\textbf{x}\in \textbf{A}^2}
\Ex \left|L_N^{-1}\sum_{u=1}^{L_N} \Ind_{(a^{I}_u = x_1, a^{J}_u = x_2)}-A^{-2}\right|. \label{166}
\end{eqnarray}
By Lemma~\ref{L4} and~\eqref{CNu}, the two names $\ba^I, \ba^J$ 
are independent uniform on $\bA^{L_N}$ outside an event 
of probability $\textrm{O}(1/N)$. Under this event, we bound
the total variation distance appearing in $\Delta_N\ta{2}$ by 1, leading to the
second summand in the bound~\eqref{089}. 
If the two names are independent, then 
because the sum below has Binomial$(L_N,A^{-2})$ distribution with variance $< L_NA^{-2}$,
\begin{equation}
 \Ex \left| 
L_N^{-1} \sum_{u=1}^{L_N} \Ind_{(a^{I}_u = x_1, a^{J}_u = x_2)} - A^{-2} \right| \leq L_N^{-1/2} A^{-1},
\label{LN1}
\end{equation}
which contributes the first summand in the bound~\eqref{089}.

The proof of Lemma~\ref{L3a} for general $k$ is similar.
 Taking $I_1,\ldots, I_k$ independent
and uniform on the set 
$\{1\leq i_1 <\cdots<i_k\leq n\}$, we have the analog of (\ref{166}):
\bal{
\Delta_N\ta{k}\leq \sfrac{1}{2}\max_{2 \le n \le N} \sum_{\textbf{x}\in \textbf{A}^k}
\Ex \left|L_N^{-1}\sum_{u=1}^{L_N} \Ind_{(a^{I_1}_u = x_1,\ldots, a^{I_k}_u = x_k)}-A^{-k}\right|. 
}
 The names $\ba^{I_1},\ldots, \ba^{I_k}$ are independent outside of the ``bad" event that
 some pair within $k$ random vertices 
have the same $C^N_1(\cdot)$ value. But the probability of this bad event
is bounded by $\binom{k}{2}$ times the chance for a given pair,
which, after applying Lemma~\ref{L4}, leads to the second summand of the bound~\eqref{089}.
And the upper bound for the term analogous to~\eqref{LN1}
becomes $L_N^{-1/2} A^{-k/2}$.
\qed

\subsection{Structure of the directed sparse \ER\ graph}
\label{sec:sparseER}
In order to prove Lemma~\ref{L4} and later results,
we study the original name variables $C_u^N(i)$ defined at~\eqref{CNu}.
It will first help to collect some facts about the structure of a directed sparse \ER\ random graph.
Write (omitting the dependence on~$N$)
\[ \TT_n = \{1 \le j \le n: \ \exists g \ge 0 \mbox{ and a path } 
n = v_0 > v_1 > \ldots > v_g = j \mbox{ in } \GG_N \} . \]
We visualize $\TT_n$ as the vertices of the tree of descendants of $n$
although it may not be a tree. The next result collects two facts about the 
structure of $\TT_n$ including that for large $N$ it is a tree with high probability.
\begin{Lemma}\label{LTn}
For $m<n$ and $\TT_n$ as above,\\
(a) $\Pr( \TT_n \cap \TT_m \ne \emptyset) \le \frac{(\alpha e^\alpha)^2 
+ \alpha e^\alpha}{N} $. \\
(b) $\Pr( \TT_n \mbox{ is not a tree }) \le \frac{ (\alpha e^\alpha)^3}{2N}$.
\end{Lemma}
\proof
First note that
for $1 \le j < n \le N$ the mean number of decreasing paths from 
$n$ to $j$ of length $g \ge 1$ equals 
$\binom{n-j-1}{g-1}  (\alpha/N)^g$.  
Because $n-j-1 \le N$, this is bounded by 
$\frac{\alpha}{N} \ \frac{\alpha^{g-1}}{(g-1)!}$, 
and summing over $g$ gives
\begin{equation}
 \Pr(j \in \TT_n) \le
\Ex (\mbox{number of decreasing paths from $n$ to $j$}) \le 
\alpha e^{\alpha}/N .
\label{Edec}
\end{equation}
We break the event $\TT_n \cap \TT_m\ne \emptyset$
into a disjoint union according to the largest element in the intersection:
$\max \TT_n\cap \TT_m=j$ for $j=1,\ldots,m$.
Now note for $j\leq m-1$, we can write
\baln{
\Pr (\max \TT_n\cap \TT_m=j)\leq \Ex \sum_{x_n^j,y_m^j}
 \Ind_{(x_n^j \mbox{ is path in } \GG_N)}
\Ind_{(y_m^j  \mbox{ is path in } \GG_N)} ,
\label{11}
}
where the sum is over edge-disjoint decreasing paths $x_n^j$ from $n$ to $j$
and $y_m^j$ from $m$ to $j$. Since the paths are edge-disjoint,
the indicators appearing in the sum~\eqref{11} are independent and so we find
\bal{
\Pr  (\max \TT_n\cap \TT_m=j)&\leq \sum_{x_n^j,y_m^j}
\Pr(x_n^j \mbox{ is path in } \GG_N) 
\Pr(y_m^j  \mbox{ is path in } \GG_N) \\
	&\leq \sum_{x_n^j}\Pr(x_n^j \mbox{ is path in } \GG_N) \sum_{y_m^j}\Pr(y_m^j  \mbox{ is path in } \GG_N) \\
	&\leq (\alpha e^{\alpha}/N)^2;
}
where the sums
in the second line are over all paths from $n$ (respectively m) to $j$, and the final
inequality follows from~\eqref{Edec}.
Now part~(a) of the lemma follows by summing over $j<m$ and adding the
corresponding bound~\eqref{Edec} for the case $j=m$.

Part (b) is proved in a similar fashion.  
If $\TT_n$ is not a tree then for some $j_2 \in \TT_n$ and some 
$j_1 < j_2$ there are two edge-disjoint paths from $j_2$ to $j_1$. 
For a given pair $(j_2,j_1)$ the mean number of such path-pairs is bounded by
$\Pr (j_2 \in \TT_n) \times (\alpha e^{\alpha}/N)^2$. 
By (\ref{Edec}) this is bounded by $(\alpha e^{\alpha}/N)^3$, and summing over pairs $(j_2,j_1)$ gives the stated bound.
\qed

\smallskip\noindent
{\bf Remark.} Note that for part (a) of the lemma we could 
also appeal to the more sophisticated 
inequalities of~\cite{MR799280} concerning disjoint occurrence of events,
which would give the stronger bound $\Pr  (\max \TT_n\cap \TT_m=j)\leq\Pr(j \in \TT_m) \times \Pr(j \in \TT_n)$.

\paragraph{Proof of Lemma \ref{L4}.}
Lemma~\ref{L4} follows from Lemma~\ref{LTn}(a)
and the observation that 
$\{C_1^N(i)=C_1^N(j)\}\subseteq\{\TT_i\cap\TT_j\ne \emptyset\}$.

\subsection{Making the vertex labels distinct}
\label{sec:unord}
In the ordered model studied above, the vertex-names are $(n,\ba^n), 1 \le n \le N$.  In order to study the
unordered model described at the start of Section~\ref{sec:hard}, 
we first must address the fact that  
the vertex-names $\ba^n, 1 \le n \le N$ may not be distinct.

\begin{Lemma}
\label{L7}
Let $\GG_N$ be random graphs-with-vertex-names, where 
(following our standing assumptions)
the names have length $\log N/\log A\leq L_N=O(\log N)$, and suppose that
for some deterministic sequence $k_N=o(N)$, the number of vertices that have non-unique names in $\GG_N$, 
say $V_N$, satisfies
$\Pr (V_N\geq k_N)=o(1)$.
Let $\GG^*_N$ be a modification with unique names obtained by re-naming some or all of the non-uniquely-named vertices.  Then 
$| \ent(\GG^*_N) - \ent(\GG_N)| = o(N \log N)$.
\end{Lemma}
\begin{proof}
The chain rule \eqref{H-condit} implies that
\bal{
\ent(\GG^*_N)& \le \ent(\GG_N)+\Ex \ent(\GG^*_N | \GG_N),
}
so we want to show that $\Ex \ent(\GG^*_N | \GG_N)$ is $o(N \log N)$.
Considering the number of ways of relabeling $V_N$ vertices, 
\bal{
\Ex \ent(\GG^*_N | \GG_N)&\leq \Ex \log \left(V_N!\binom{A^{L_N}}{V_N}\right), \\
	&\leq \Ex (\log A^{V_N L_N}) \Ind_{(V_N<k_N)} +\Ex(\log A^{L_N V_N}) \Ind_{(V_N\geq k_N)}, \\
	&\leq \log(A) L_N [k_N +  N \Pr(V_N\geq k_N)]=o(N\log N),
}  
as desired. \qed
\end{proof}

\noindent{\bf Remark.} The analogous lemma holds if instead we replace the labels of any random subset of vertices of $\GG_N$
to form $\GG_N^*$, provided the subset size satisfies the same assumptions as $V_N$.

\begin{Lemma}\label{LVN}
For $\GG^{ord}_N$ and $\GG_N$,
\baln{\label{evn}
\Ex |\{n \, : \, \ba^n = \ba^m \mbox{ for some } m \ne n\}| = o(N).
}
\end{Lemma}
\proof 
As in Lemma~\ref{L4}, the proof is based on studying
the originating vertex $C_i(n)$ (now dropping the notational dependence on $N$)
of the letter ultimately 
copied to coordinate $i$ of vertex $n$ through the ``trees" $\TT_n$.
Given $\TT_n$, the copying mechanism that determines the name $\ba^n$ 
evolves independently for each coordinate, and this implies 
the {\em conditional independence property:} for $1 \le m < n \le N$, 
the events $\{C_i(n) = C_i(m)\}, 1 \le i \le L_N$ are
conditionally independent given $\TT_m$ and $\TT_n$.
Because
\bal{ 
&\Pr ( a^n_i = a^m_i \vert \TT_n, \TT_m) \\
&\qquad=\Pr ( C_i(n) = C_i(m) \vert \TT_n, \TT_m) + 
\sfrac{1}{A} \Pr ( C_i(n) \ne C_i(m) \vert \TT_n, \TT_m) 
}
the conditional independence property implies
\begin{equation}\label{c-i}
\begin{split}
&\Pr ( \ba^n = \ba^m \vert \TT_n, \TT_m) \\
&\quad=\left[ \Pr ( C_1(n) = C_1(m) \vert \TT_n, \TT_m) + 
\sfrac{1}{A} \Pr ( C_1(n) \ne C_1(m) \vert \TT_n, \TT_m)\right]^{L_N}.
\end{split}
\end{equation}
Now we always have $C_1(n) \in \TT_n$ so trivially
\begin{equation}
 \Pr ( C_1(n) = C_1(m) \vert \TT_n, \TT_m) = 0 
\mbox{ on } \TT_n \cap \TT_m = \emptyset . 
\label{C1C}
\end{equation}
We show below that when the sets do intersect we have
\baln{\label{LC1n}
\Pr ( C_1(n) = C_1(m) \vert \TT_n, \TT_m) \le \sfrac{1}{2} 
\mbox{ on } \{ \mbox{$\TT_n$ and $\TT_m$ are trees} \}.
}
Assuming~\eqref{LC1n}, since
$A \ge 2$, for $p \le 1/2$, we have $p + (1-p)/A \le 3/4$,
and now combining (\ref{c-i}, \ref{C1C}, \ref{LC1n}), we find
\[
\Pr ( \ba^n = \ba^m \vert \TT_n, \TT_m) 
\le (\sfrac{3}{4})^{L_N} \Ind_{( \TT_n \cap \TT_m \ne \emptyset )}  
\mbox{ on }
\{ \mbox{$\TT_n$ and $\TT_m$ are trees} \}.
\]
Now take expectation, appeal to part~(a) of Lemma~\ref{LTn}, and sum over $m$ to conclude 
\[
\Pr (\mbox{$\TT_n$ is a tree}, \ba^n = \ba^m \mbox{ for some } m \ne n \mbox{ for which 
$\TT_m$ is a tree} ) \]
\[ \le (\sfrac{3}{4})^{L_N} \ ((\alpha e^\alpha)^2 
+ \alpha e^\alpha)  
\to 0.
\] 
Now any $n$ for which the name $\ba^n$ is not unique is either in the set 
of $n$ defined by the event above, or in one of the two following sets:
\[ \{n: \mbox{$\TT_n$ is not a tree} \} \]
\[ \{n: \mbox{$\TT_n$ is a tree, $ \ba^n = \ba^m$ for some  $m\ne n$ for which $\TT_m$ is not a tree
} \]
\vspace*{-0.25in}
\[ \mbox{ but  $\ba^n \ne\ba^m$ for all  $m\ne n$ for which $\TT_m$ is a tree} \}. \]
The cardinality of the final set is at most the cardinality of the previous set, 
which by part~(b) of Lemma~\ref{LTn} has expectation 
$O(1)$.
Combining these bounds gives (\ref{evn}).

It remains only  to prove~\eqref{LC1n}.
For $v \in \TT_n$ write $R_v(n)$ for the event that the path of copying 
of coordinate $1$ from $C_1(n)$ to $n$ passes through $v$.
We may assume there is at least one edge from $n$ into $[1,n-1]$ 
(otherwise we are in the setting of (\ref{C1C})).  Given $\TT_n$, the chance that
vertex $n$ adopts the label of any given neighbor in $\TT_n$ is bounded by $1/2$,  we see
\begin{equation}
 \Pr(R_v(n) | \TT_n) \le \sfrac{1}{2}, \quad v \in \TT_n . 
\label{Rvn}
\end{equation}
Similarly by (\ref{C1C}) we may assume $ \TT_n \cap \TT_m \ne \emptyset$.
By hypothesis $\TT_n$ and $\TT_m$ are trees, and so there is a subset 
$\MM \subseteq \TT_n \cap \TT_m$ of ``first meeting" points $v$ with the property that the path from $v$ to $n$ in $\TT_n$ does not meet the 
path from $v$ to $m$ in $\TT_m$ and
\[ \{  C_1(n) = C_1(m)\} = \cup_{v \in \MM}  \left[R_v(n) \cap R_v(m) \right] \]
with a disjoint union on the right.  So
\begin{equation}
 \Pr ( C_1(n) = C_1(m) \vert \TT_n, \TT_m) 
= \sum_{v \in \MM} \Pr(R_v(n) \vert \TT_n) \times \Pr(R_v(m) \vert \TT_m) . \label{CCm}
\end{equation}
Now $v \to \Pr(R_v(n) \vert \TT_n)$ and $v \to \Pr(R_v(m) \vert \TT_m)$ 
are sub-probability distributions on $\MM$ and the former satisfies (\ref{Rvn}).  Now (\ref{CCm}) implies (\ref{LC1n}).
\qed

\subsection{The unordered model and its entropy rate}
\label{sec-unordered}
The model we introduced as $\GG_N$ in
section \ref{sec:hybrid_model} 
does not quite fit our default setting because the vertex-names will 
typically not be all distinct.  
However, if we take the ordered model $\GG_N^{ord}$ and then 
 arbitrarily 
rename the non-unique names, to obtain a model $\GG_N^{ord*}$ say,
then Lemmas \ref{L7} and \ref{LVN} imply that only a proportion $o(1)$ of vertices 
are renamed and the entropy rate is unchanged:
\[ (\GG_N^{ord*}) \mbox{ has entropy rate (\ref{hybrid-rate}) } . \]
Now we can ``ignore the order", that is replace the names 
$\{(n,\ba^n)\}$ by the now-distinct names $\{\ba^n\}$, 
to obtain a model $\GG_N^*$, say.  
In this section we will obtain the entropy rate formula for 
$(\GG_N^*)$ as 
\baln{
\mbox{(entropy rate for $\GG^*_N$)}= \mbox{ (entropy rate for $\GG^{ord*}_N$)}  - 1. \label{aut}
}

The remainder of this section is devoted to the proof of~\eqref{aut}.
Write $\HH_N$ for the \ER\ graph arising in the construction of
$\GG^{ord*}_N$; that is, each vertex $n+1$ is linked to each earlier 
vertex $i$ with probability $\alpha/N$, and we regard the created edges as  directed edges 
$(n+1,i)$. 
Now delete the vertex-labels; consider the resulting graph $\HH^{unl}_N$ 
as a random unlabelled directed acyclic graph.  
Given a realization of $\HH^{unl}_N$ there is some number 
$1 \le M(\HH^{unl}_N) \le N!$ of possible vertex orderings consistent with the edge-directions of the realization.
\begin{Lemma}\label{L5b}
In the notation above,
\baln{
\ent(\GG^{ord*}_N) = \ent(\GG_N^*)  + \Ex \log  M(\HH^{unl}_N). \label{009}
}
\end{Lemma}

\smallskip  \noindent \proof
According to the chain rule (\ref{H-condit}),
\bal{
\ent(\GG^{ord*}_N)=\ent(\GG_N^*)+\Ex \ent(\GG^{ord*}_N|\GG_N^*).
}
We only need to show
\bal{
\ent(\GG^{ord*}_N|\GG_N^*)=\log  M(\HH^{unl}_N),
}
which follows from two facts: given $\GG_N^*$, all possible vertex orderings
consistent with the edge-directions of $\HH^{unl}_N$ are equally likely
and there are $M(\HH^{unl}_N)$ of these orderings. The latter fact is 
obvious from the definition and to see the former,
consider two such orderings; there is a permutation taking one to the other.
Given a realization of $\GG^{ord*}_N$ associated with the realization of 
$\HH^{unl}_N$, applying the same permutation gives a different realization of 
 $\GG^{ord*}_N$ associated with the same realization of 
$\HH^{unl}_N$.  These two realizations of  $\GG^{ord*}_N$ have the same probability, and map to the same element of
$\GG_N^*$, and (here we are using that the second part of the labels are all distinct)
this is the only way that different  realizations of  $\GG^{ord*}_N$ 
can map to the same element of $\GG_N^*$. 
\qed

So it remains only to prove
\begin{Proposition}
\label{P1}
\[  \Ex \log  M(\HH^{unl}_N)  \sim N \log N . \]
\end{Proposition}

%
%

\smallskip  \noindent \proof 
Choose $K_N \sim N^\eps$ for small $\eps > 0$ and partition the labels 
$[1,N]$ into $K_N$ consecutive intervals 
$I_1, I_2,\ldots$ each containing $N/K_N$ labels.  
Consider a realization of the (labeled) \ER\ graph $\HH_N$. 
The number $V_i$ of edges with both end-vertices in $I_i$ has 
Binomial($ \binom{N/K_N}{2}, \alpha/N$) distribution with mean 
$\sim \frac{\alpha N}{2K_N^2}$, 
and from standard large deviation bounds 
(e.g.~\cite{MR2248695} Theorem 2.15)
\[ \Pr ( V_i \le \sfrac{\alpha N}{K_N^2}, \mbox{ all } 1 \le i \le K_N) \to 1 . \]
For a realization $\HH_N$ satisfying these inequalities we have 
\[ M(\HH^{unl}_N) \ge \left( 
\left(\frac{N}{K_N} - \frac{2\alpha N}{K_N^2} \right) !
 \ \right)^{K_N} . \]
This holds because we can create permutations consistent with 
$\HH^{unl}_N $ by, on each interval $I_i$, first placing the 
(at most $ \frac{2\alpha N}{K_N^2} $) labels involved in the edges with both ends in $I_i$ in increasing order, then placing the remaining labels in arbitrary order.
So
\begin{eqnarray*}
 \Ex \log  M(\HH^{unl}_N) &\ge& (1 - o(1)) \log  
\left( 
\left(\frac{N}{K_N} - \frac{2\alpha N}{K_N^2} \right) !
 \ \right)^{K_N} \\
&\sim& K_N \times \sfrac{N}{K_N} \log \sfrac{N}{K_N} \\
&\sim& (1-\eps) N \log N
\end{eqnarray*} 
establishing Proposition~\ref{P1}.
\qed

\smallskip \noindent{\bf{Remark.}}
Proposition~\ref{P1} and Lemma~\ref{L5b}
are in the spirit of the 
graph entropy literature, but we could not
find these results there. As discussed in Section~\ref{sec:concept}, this literature
is largely concerned with the complexity of the structure of an unlabeled graph, or
in the case of~\cite{SZP}, the entropy of probability distributions on unlabeled graphs.
A quantity of interest in these settings is the ``automorphism group" of the graph which
is closely related to $M(\HH^{unl}_N)$ here.
For example, an analog of~\eqref{009} is 
shown in Lemma~1 of~\cite{SZP} and Theorem~1
there uses this lemma to
relate the entropy rate between an \ER\ graph on $N$ vertices
with edge probabilities $p_N$ with distinguished vertices and that of
the
same model where the vertex labels are ignored. 
Their result is very close to Proposition~\ref{P1}, but~\cite{SZP} only
considers
edge weights $p_N$ satisfying $N p_N/\log(N)$ bounded away from zero, which
falls outside our setting.

\section{Open problems}
\label{sec:final}
Aside from the (quite easy) Lemmas \ref{L1} and \ref{L2}, our results concern specific models.  Are there interesting ``general" results in this topic?  
Here are two possible avenues for exploration.

Given a random graph-with vertex-names $\GG = \GG_N$, there is an associated 
random unlabeled graph $\GGunl$ and an associated random unordered set of names $\Names$, and obviously 
\[ \ent(\GG) \ge \max(\ent(\GG^{unl}), \ent(\Names)) . \]
Lemmas \ref{L1} and \ref{L2}, applied conditionally as indicated in Section \ref{sec:zero_rate},
give sufficient conditions for 
$\ent(\GG) = \ent(\GG^{unl})$ or $\ent(\GG) = \ent(\Names)$. 
In general one could ask
``given $\GG^{unl}$ and $\Names$, how random is the assignment of names to vertices?"
The standard notion of {\em graph entropy} enters here, as a statistic of the ``completely random" assignment, so the appropriate conditional graph entropy within a model constitutes a measure of relative randomness.  
Another question concerns measures of strength of association of names across edges.
One could just take the space $\bA^{L_N}$ of possible names, consider
the empirical distribution across edges $(v,w)$ of the pair  of names
$(\ba(v), \ba(w))$ as a distribution on the product space $\bA^{L_N} \times \bA^{L_N}$ 
and compare with the  product measure using
some quantitative measure of dependence.
But neither of these procedures quite gets to grips with the issue of finding
conceptually interpretable quantitative measures 
of dependence between graph structure and name structure, which we propose as an open problem.
  
A second issue concerns ``local" upper bounds for the entropy rate.
In the classical context of sequences $X_1,\ldots,X_n$ from $\bA$, 
an elementary consequence of subadditivity  
is that (without any further assumptions) one can upper bound 
$\ent(X_1,\ldots,X_n)$ in terms of the ``size-$k$ random window" entropy
\[ \EE_{n,k} := \ent(X_U, X_{U+1}, \ldots,X_{U+k-1}); \quad U 
\mbox{ uniform on } [1,n-k+1]   \]
and this is optimal in the sense that  for a stationary ergodic sequence the 
``global" entropy rate is actually equal to the quantity 
\[ \lim_{k \to \infty} \lim_{n \to \infty} k^{-1} \EE_{n,k} \]
arising from this ``local" upper bound.  
In our setting we would like some analogous result saying that, for the entropy 
$\EE_{N,k}$ of the restriction of $\GG_N$ to some ``size-$k$" neighborhood of a random vertex, there is always an upper 
bound for the entropy rate $c$ of the form
\[ c \le \lim_{k \to \infty} \lim_{N \to \infty}  \frac{ \EE_{N,k}}{k \log N}  \]
and that this is an equality under some ``no long-range dependence" condition analogous to 
ergodicity.  But results of this kind seem hard to formulate, because
of the difficulty in specifying which vertices and edges are to be included in the ``size-$k$" neighborhood.

\paragraph{Acknowledgement.}
The hybrid model arose from a conversation with Sukhada Fadnavis.


\section{Appendix}

\paragraph{Small worlds model: $0<\gamma<2$.}
Here we complete the analysis of the  graph entropy rate 
in the ``small worlds" model of Section \ref{sec:SW}. 
First we show that for $a$ as in \eqref{sw-d2}, the average degree tends to a constant.
For $D_u=(n-1)/2+b$ and $D_l=(n-1)/2-b$, where $b$ is constant with respect to $N$ (and chosen
large enough for the inequalities below to hold), we find using \eqref{sw1}
that
\begin{eqnarray}
 8a \int_0^{\pi/4}\int_1^{D_l\sec(\theta)} r^{-\gamma+1} dr d\theta
&\leq  & \Ex D(v) -4 \nonumber\\
&\leq & 8a \int_0^{\pi/4}\int_1^{D_u\sec(\theta)} r^{-\gamma+1} drd\theta+\sfrac{4a}{2^{\gamma/2}}, \nonumber\\
\sfrac{8aD_l^{2-\gamma}}{2-\gamma}\int_0^{\pi/4} \sec^{2-\gamma}(\theta) d\theta -\sfrac{2a\pi}{2-\gamma}
&\leq&  \Ex D(v) -4 \label{44} \\
&\leq & \sfrac{8aD_u^{2-\gamma}}{2-\gamma}\int_0^{\pi/4} \sec^{2-\gamma}(\theta) d\theta -\sfrac{2a\pi}{2-\gamma}+\sfrac{4a}{2^{\gamma/2}}. \nonumber
\end{eqnarray}

Taking $a$ and $\kappa_\gamma$ as in \eqref{sw-d2}, the inequalities above imply
$\Ex D(v)\to4+\alpha$.

To show the entropy rate is as claimed, take $N$ large enough to make $a<1/2$, so that
$\EE(ar^{-\gamma})$ is a decreasing function of $r$ for $r>1$. 
Using the inequality $-(1-x)\log(1-x)\leq x$ for $0<x<1$,   
\baln{ 
\frac{-8a}{\log(N)}&\int_0^{\pi/4}\int_1^{D_l\sec(\theta)} \EE (ar^{-\gamma}) r dr d\theta \notag \\
&\qquad\sim\frac{-8a}{\log(N)}\int_0^{\pi/4}\int_1^{D_l\sec(\theta)} r^{-\gamma+1} \log(a r^{-\gamma}) dr d\theta, \label{3}
}
and we will show \eqref{3} tends to $\alpha$ as $N\to\infty$.
From this point, similar arguments show the same is true with $D_l$ replaced by $D_u$, so that 
following the arguments that established the convergence of the average degree and using~\eqref{2}, we find
\bal{
\lim_{N\to\infty} \frac{\ent(\GG_N)}{N\log(N)}=\frac{\alpha}{2}, 
}
as desired. To obtain the claimed asymptotic, note that we can write~\eqref{3} as
\baln{
&-8a \log(a)\int_0^{\pi/4}\int_1^{D_l\sec(\theta)} r^{-\gamma+1} dr d\theta \label{5} \\
&\qquad+ 8a\gamma \int_0^{\pi/4}\int_1^{D_l\sec(\theta)} r^{-\gamma+1} \log(r) dr d\theta.  \label{6}
}
From \eqref{44} above and the definition \eqref{sw-d2} of $a$, it is easy to see
that \eqref{5} is 
\baln{ \label{64}
\alpha(1-\gamma/2)\log(N)+\textrm{o}(\log(N)) \mbox{ as \,} N\to\infty.
}
Now, making the substitution $u=r^{2-\gamma}$, \eqref{6} is equal to
\bal{
&\frac{8a\gamma}{(2-\gamma)^2} \int_0^{\pi/4}\int_1^{D_l^{2-\gamma} \sec^{2-\gamma}(\theta)}\log(u) du d\theta \\
&\qquad=\frac{8a\gamma}{(2-\gamma)^2} \int_0^{\pi/4} D_l^{2-\gamma} \sec^{2-\gamma}(\theta) [\log( D_l^{2-\gamma} \sec^{2-\gamma}(\theta))-1] d\theta.
}
After simplification, the only term that is not $\textrm{o}(\log(N))$ is
\bal{
&\frac{8a\gamma}{(2-\gamma)^2}  D_l^{2-\gamma} \log( D_l^{2-\gamma}) \int_0^{\pi/4}\sec^{2-\gamma}(\theta) d\theta,
}
which after simplification is equal to
\baln{ \label{65}
\frac{\alpha \gamma}{2}\log(N)+\textrm{o}(\log(N)).
}
Combining \eqref{64} and \eqref{65} with \eqref{5} and \eqref{6} implies \eqref{3} tends to $\alpha$ as $N\to\infty$.

\end{document}